\documentclass[11pt,reqno]{amsart}
\usepackage{amsmath,amssymb}
\usepackage[all]{xy}
\usepackage{hyperref}
\usepackage{verbatim}
\theoremstyle{plain}
\newtheorem{thm}{Theorem}[section]
\newtheorem{cor}[thm]{Corollary}
\newtheorem{lmm}[thm]{Lemma}
\newtheorem{prp}[thm]{Proposition}
\newtheorem{que}[thm]{Question}

\theoremstyle{definition}

\newtheorem{rem}[thm]{Remark}

\theoremstyle{definition}

\numberwithin{equation}{section}

\usepackage{hyperref}
\hypersetup{
	colorlinks,
	citecolor=red,
	filecolor=black,
	linkcolor=blue,
	urlcolor=black
}

\providecommand \xii{\mathcal{L}}

\providecommand \A{\mathrm{A}}
\providecommand \B{\mathrm{B}}
\providecommand \R{\mathcal{R}}
\providecommand \J{\mathcal{J}}
\providecommand \F{\mathrm{F}}
\providecommand \G{\mathrm{G}}

\providecommand \mkdpt{\lambda}
\providecommand \mkdptt{\gamma}
\providecommand \fmkdpt{\alpha}

\begin{document}

\title[Rational cuspidal curves on del-Pezzo surfaces]{Rational cuspidal curves on
del-Pezzo surfaces}

\author[I. Biswas]{Indranil Biswas}

\address{School of Mathematics,
Tata Institute of fundamental research, Homi Bhabha road, Mumbai 400005, India}

\email{indranil@math.tifr.res.in}

\author[S. D'Mello]{Shane D'Mello}

\address{School of Mathematics, Tata Institute of Fundamental
Research, Homi Bhabha Road, Mumbai 400005, India}

\email{shaned@math.tifr.res.in}

\author[R. Mukherjee]{Ritwik Mukherjee}

\address{School of Mathematics, Tata Institute of Fundamental
Research, Homi Bhabha Road, Mumbai 400005, India}

\email{ritwikm@math.tifr.res.in}

\author[V. P. Pingali]{Vamsi P. Pingali}

\address{Department of Mathematics,
412 Krieger Hall, Johns Hopkins University, Baltimore, MD 21218, USA}

\email{vpingali@math.jhu.edu}

\subjclass[2010]{14N35, 14J45}

\date{}

\begin{abstract}
We obtain an explicit formula for the number of rational cuspidal curves of a given
degree on a del-Pezzo surface that pass through an appropriate number of generic
points of the surface. This enumerative problem is expressed as an Euler class
computation on the moduli space of curves. A topological method is employed in
computing the contribution of the degenerate locus to this Euler class.
\end{abstract}

\maketitle

\tableofcontents

\section{Introduction}
\label{introduction}

Enumerative Geometry of rational curves in $\mathbb{P}^2_{\mathbb C}$ is a classical 
question. However, a formula for the number of degree $d$ rational curves in
$\mathbb{P}^2_{\mathbb C}$ through $3d-1$ generic points was unknown until the early
$90^{' \textnormal{s}}$ when Ruan--Tian \cite{RT} and Kontsevich--Manin \cite{KoMa}
obtained a formula for it. More generally, they gave an explicit answer to the
following question:

 \begin{que}\label{qu}
Let  $X$ be a complex del-Pezzo surface, and let $\beta \,\in\, H_2(X;\, \mathbb{Z})$
be a given homology class.
 What is the number of rational degree $\beta$-curves in $X$ that pass through
 $\langle c_1(TX), ~\beta \rangle -1$ generic points?
 \end{que}

Fixing $X$, the number in Question \ref{qu} will be denoted by $N_\beta$.

A natural generalization to the above question is to ask how many rational curves 
are there of a given degree, that pass through the right number of generic points 
and have a specific singularity.

The main result we obtain here is the following:

\begin{thm}
\label{main_thm}
Let $X$ be $\mathbb{P}^2$ blown up at $k$-points with $k\, \leq\, 8$, and let
$$\beta\,:=\, d L -m_1 E_1 - \ldots -m_k E_k\,\in\, H_2(X;\, \mathbb{Z})$$ be a
homology class, where $L$ denotes the homology class
of a line, $\{E_i\}_{i=1}^k$ are the exceptional divisors and $m_i\, \geq\, 0$. Denote
\[
x_i\,:=\, c_i(TX) \qquad\textnormal{and} \qquad \delta_{\beta} \,:=\,
\langle x_1, \beta \rangle -1\, ,
\]
where $c_i$ denotes the $i$-th Chern class.
If $N_{\beta - 3L} > 0$, then the
number of rational degree $\beta$-curves in $X$ that pass through $\delta_{\beta}-1$ generic
points and have a cusp, is given by
\begin{align}
C_{\beta} &= \Big(x_2([X]) - \frac{x_1\cdot x_1}{\beta\cdot x_1} \Big) N_{\beta} +
\sum_{\substack{\beta_1+ \beta_2= \beta, \\ \beta_1, \beta_2 \neq 0}} \binom{\delta_{\beta}-1}{\delta_{\beta_1}}
N_{\beta_1} N_{\beta_2} (\beta_1 \cdot \beta_2) \Big(
\frac{(\beta_1 \cdot x_1) (\beta_2 \cdot x_1)}{2 (\beta \cdot x_1)} -1 \Big)\, , \label{c_beta}
\end{align}
where ``$\cdot$'' denotes topological intersection.
\end{thm}

Since the numbers $N_{\beta}$ are known using
the algorithm described in
\cite{KoMa} and \cite{Pandh_Gott}, the number $C_{\beta}$ is computable
using \eqref{c_beta}.
We have written a $\mbox{C++}$ program
that implements \eqref{c_beta}
and computes  $C_{\beta}$ for a given $\beta$.
The program is available on our web page
\[ \textnormal{\url{https://www.sites.google.com/site/ritwik371/home}}. \]

We need the condition $N_{\beta - 3L} >0$
in order to prove that the space of rational curves having 
exactly one genuine cusp is non-empty, and also to prove a transversality
result (Section \ref{transverse_proof}). 
However, based on the numerical evidence
we also expect the formula to be valid even when
\begin{equation}\label{f3}
N_{\beta - 3L}\,=\,0\, .
\end{equation}
The condition $N_{\beta - 3L} \,>\,0$
is a sufficient to prove transversality; it may not be
necessary. It should be mentioned that \eqref{f3} is often vacuously true 
if there do not exist any cuspidal curves in a given class $\beta$.

When $X\,:= \,\mathbb{P}^2$, Pandharipande, \cite{Rahul1},
obtains a formula for $C_{\beta}$ using an
algebro-geometric method. Theorem \ref{main_thm} is consistent with his results.
Furthermore,
it is easy to see by direct geometric arguments that
\begin{equation}\label{low_degree_blow_up}
C_{dL + \sigma_1 E_1 + \sigma_2 E_2 + \ldots + \sigma_k E_k} \,=\,
C_{dL}\, ,
\end{equation}
where each of the $\sigma_i$ is $-1$ or $0$. We have verified
\eqref{low_degree_blow_up}  for several cases using the above mentioned program.

In \cite{KoMa}, the authors strictly speaking give a formula to compute the genus
$0$ Gromov--Witten invariants of the del-Pezzo surfaces. A priori, these numbers need
not be the same as $N_{\beta}$ (since the Gromov--Witten invariants are not always
enumerative). It is established in \cite{Pandh_Gott} that the numbers obtained in
\cite{KoMa} are indeed enumerative, meaning they are actually equal to $N_{\beta}$.

Theorem \ref{main_thm} also happens to be true when $X:= \mathbb{P}^1 \times \mathbb{P}^1$.
In \cite{JKoch2}, Koch obtains a formula for $C_{\beta}$ using an
algebro-geometric method, which is consistent with \eqref{c_beta}.
In order to
keep the exposition here more streamlined, we decided to omit working out this case
separately. The arguments given in Section \ref{degenerate_contrib} go through without
any essential change; the arguments in Section \ref{transverse_proof} need to be modified
slightly to address the case of $\mathbb{P}^1 \times \mathbb{P}^1$.

\subsection{Further comments}

In \cite{Diaz}, the authors express
the divisor of cuspidal curves on surfaces in terms of other divisors. In \cite{Rahul1},
the author
shows how to
compute the latter on $\mathbb{P}^n$ and hence on any smooth algebraic variety.
This should in principle produce our formula \eqref{c_beta}. However, one of the advantages of the
topological
method presented here is that it readily applies in enumerating
rational cuspidal curves
on complex manifolds of dimension greater than two.
Furthermore, the method also applies to enumerating
rational curves with more degenerate
singularities.
The method presented here is an extension of the
method the author applies in \cite{g2p2and3} to enumerate rational cuspidal curves
in $\mathbb{P}^n$.
In \cite{genus_three_zinger}, the author further extends the
method and obtains a formula for the number of
rational curves in $\mathbb{P}^2$
with an $E_6$ singularity. (In local coordinates an $E_6$-singularity is
given by the equation $y^3 + x^4 \,=\,0$.)
Using the results obtained here, we expect that
we will be able to extend the results of the author
in \cite{genus_three_zinger}
(for $E_6$ singularities)
to del-Pezzo surfaces.
We also expect that we will be able to extend our results for rational cuspidal
curves on higher dimensional complex manifolds (such as products of projective spaces or
$\mathbb{P}^n$ blown up at certain number of points).
Finally, we should mention that we expect to apply the results obtained here to
enumerate genus $g$  
curves with a fixed complex structure on del-Pezzo surfaces.
The authors in \cite{Rahul_genus_1} and
\cite{ionel_genus1} obtain a formula for the number of genus
one curves in $\mathbb{P}^2$
with a fixed complex structure. The author in \cite{g2p2and3} and 
\cite{genus_three_zinger} obtains a formula for the number of 
genus two and three curves with a fixed complex structure in $\mathbb{P}^2$.  
We expect to extend the results 
in \cite{ionel_genus1}, \cite{g2p2and3} and \cite{genus_three_zinger} 
to del-Pezzo surfaces, since we now have a good understanding
of the intersection of tautological classes on the moduli space of curves
on del-Pezzo surfaces (Section \ref{ITC}).

\section{Notation}

Consider the rational curves on smooth complex del-Pezzo surface $X$
representing $\beta\,\in\, H^2(X,\, {\mathbb Z})$ and equipped with $n$ ordered marked
points. Let $\mathcal{M}_{0,n}(X, \beta)$
denote the moduli space of equivalence classes of such curves. In other words,
$$
\mathcal{M}_{0,n}(X, \beta)\,:=\, \{ (u, y_1, \cdots, y_n) \,\in\,
\mathcal{C}^{\infty}(\mathbb{P}^1, X) \times
(\mathbb{P}^1)^n\,\mid ~ \overline{\partial}u =0, ~~u_*[\mathbb{P}^1] = \beta \}/\text{PSL}
(2, \mathbb{C})\, .
$$
with $\text{PSL}(2, \mathbb{C})$ acting diagonally on $\mathbb{P}^1\times
(\mathbb{P}^1)^n$. For any $k\,\leq \,n$, define the subspace
$$\mathcal{M}_{0,n}(X, \beta; p_1, \cdots, p_{k})\,\subset\,\mathcal{M}_{0,n}(X, \beta)$$
consisting of $n$ marked points such that the $i$-th marked point is $p_i$
for all $1\,\leq\, i\, \leq\, k$, so,
$$
\mathcal{M}_{0,n}(X, \beta; p_1, \cdots, p_{k})
\,:=\, \{ [u, y_1, \cdots, y_n] \in \mathcal{M}_{0,n}(X, \beta)\,\mid~ u(y_i) \,=\,
p_i \qquad \forall~ i \,= \,1, \cdots ,k\}\, .
$$
We define $\mathcal{M}_{0,n}^*(X, \beta)$ (respectively, $\mathcal{M}_{0,n}^*(X,
\beta; p_1, \cdots, p_k)$) to be the locus in $\mathcal{M}_{0,n}(X, \beta)$
(respectively, $\mathcal{M}_{0,n}(X, \beta; p_1, \cdots, p_k)$) of curves that are
not multiply covered. \\

We will denote $\mathcal{M}^{*}_{0,\delta_{\beta}}(X, \beta; p_{_1},
\cdots,p_{_{\delta_{\beta} -1}})$ also by $\mathcal{M}^{*}$.

Let $\overline{\mathcal{M}}_{0,n}(X, \beta)$ denote the stable map
compactification of $\mathcal{M}_{0,n}^*(X, \beta)$. Let
$$\xii_i \,\longrightarrow \, \overline{\mathcal{M}}_{0,n}(X, \beta)$$
be the universal tangent line bundle at the $i$-th marked point. More
precisely, if $$f_C\, :\, {\mathcal C}\, \longrightarrow\,
\overline{\mathcal{M}}_{0,n}(X, \beta)$$
is the universal curve with $T_{f_C}\, \longrightarrow\, {\mathcal C}$ being the relative
tangent bundle for $f_C$
and $$y_i\, :\, \overline{\mathcal{M}}_{0,n}(X, \beta)\,\longrightarrow\,
{\mathcal C}$$ is the section giving the $i$-th marked point, then
$\xii_i \,:=\, y^*_i T_{f_C}$.

\section{Euler class computation}\label{sec2}

Let us now explain how we obtain \eqref{c_beta}. The method employed follows
closely the method in \cite{g2p2and3} to compute $C_{\beta}$ when $X\,= \,\mathbb{P}^2$.\\

Evidently, $C_{\beta}$ coincides with the cardinality of the following set
$$
\{[u; y_{_1}, \cdots, y_{_{\delta_{\beta}-1}}; y_{_{\delta_{\beta}}}] \,\in\,
\mathcal{M}_{0,\delta_{\beta}}^{*}(X, \beta; p_1, \cdots, p_{_{\delta_{\beta} -1}})
\,\mid\, ~
du\vert_{y_{_{\delta_{\beta}}}} \,=\,0\}.
$$
Since the above $\delta_{\beta}-1$ points are in general position, the curve will
have a genuine cusp at the last marked point (as opposed to something more degenerate).
Furthermore, the curves will not have any other singular points (aside from the 
nodes which are just points of self intersections).
For any
$$[u; y_{_1}, \cdots, y_{_{\delta_{\beta}-1}}; y_{_{\delta_{\beta}}}] \,\in\,
\mathcal{M}^{*}_{0,\delta_{\beta}}(X, \beta; p_1, \cdots, p_{_{\delta_{\beta} -1}})\, ,
$$
the differential of $u$ at the last marked point $y_{_{\delta_{\beta}}}$
produces a section $\psi$ of the
rank two vector bundle
\[ \xii^*_{\delta_{\beta}} \otimes \textnormal{ev}_{\delta_{\beta}}^* TX
\,\longrightarrow\, \mathcal{M}_{0,\delta_{\beta}}^{*}(X, \beta)\, ,\]
where $\textnormal{ev}_i\,:\, \mathcal{M}_{0,\delta_{\beta}}^{*}(X, \beta)\,
\longrightarrow\, X$ is the evaluation at the $i$-th marked point. This
$\psi$ is transverse to the zero section (this is shown in Section
\ref{transverse_proof}).
Moreover, it has a natural extension to the compactification
$$
\overline{\mathcal{M}} \,:=\, \overline{\mathcal{M}}_{0,\delta_{\beta}}(X, \beta;
p_{_1}, \cdots, p_{_{\delta_{\beta} -1}})\, .
$$
Therefore, we have
\begin{equation}
\langle e(\xii^*_{\delta_{\beta}} \otimes \textnormal{ev}_{\delta_{\beta}}^* TX)\, ,
~[\overline{\mathcal{M}}]
\rangle \, =\, C_{\beta} + \mathcal{C}_{\partial \overline{\mathcal{M}}}
(du\vert_{y_{_{\delta_{\beta}}}})\, ,  \label{Euler_du}
\end{equation}
where $\mathcal{C}_{\partial \overline{\mathcal{M}}} (du\vert_{y_{_{\delta_{\beta}}}})$
is the contribution of the
extended section to the Euler class from the boundary $\overline{\mathcal{M}}
\setminus \mathcal{M}^{*}$.

Let us now take a closer look at \eqref{Euler_du}. First, we note that  the extended
section $du\vert_{y_{\delta_{\beta}}}$ (over $\overline{\mathcal{M}}$) vanishes only on a
finite set of points of the boundary.
It only vanishes when the curve splits as a $\beta_1$ curve
and a $\beta_2$ curve, $\beta_1, \beta_2 \,\neq\, 0$,
with the last marked point lying on a
ghost bubble. It is clear that the section vanishes on this configuration
(since taking the derivative of a constant map gives us zero and the map defined on
the ghost bubble is a constant map).
To see why the section does not vanish on any other configuration we consider all
the remaining possible cases. Suppose the curve
splits as $\beta \,=\, \beta_1 + \beta_2 + \ldots+ \beta_k$, with
$k \,\geq\, 3$ and $\beta_i \,\neq\, 0$
for all $i$. Since $\delta_{\beta_1} + \ldots + \delta_{\beta_k} \,< \,\delta_{\beta}-1$
as $k \,\geq\, 3$, such a configuration can not occur, because it will not pass through
$\delta_{\beta}-1$ generic points. Next, suppose the curve splits as
$\beta \,= \,\beta_1 + \beta_2$, with $\beta_1\, , \beta_2 \,\neq\, 0$
and the marked point lying on say the $\beta_1$ component. Then the $\beta_1$ curve is
cuspidal. Hence it can pass through $\delta_{\beta_1}-1$ general points. Since
$\delta_{\beta_1}-1 + \delta_{\beta_2}\,<\, \delta_{\beta}-1$, this configuration can not
occur. Next, we note that although the section vanishes on curves that
have singularities more degenerate than a cusp or curves that have more than one
cuspidal point, no such curve will
pass through $\delta_{\beta}-1$ general points.
Finally, we also observe  that although the section
vanishes on multiply covered curves, such curves will not pass through $\delta_{\beta}-1$
generic points. Hence, the only points on which the section vanishes are those
that split as a $\beta_1$ curve
and a $\beta_2$ curve, with the
last marked point lying on a
ghost bubble.

We show in  Section \ref{degenerate_contrib} that the extended section
vanishes with multiplicity $1$ on these boundary points. Hence, we gather that
\begin{equation}
\mathcal{C}_{\partial \overline{\mathcal{M}}} (du\vert_{y_{_{\delta_{\beta}}}}) \,=\,
1\times \frac{1}{2}\sum_{
\substack{\beta_1+ \beta_2= \beta, \\
\beta_1, \beta_2 \neq 0} }
\binom{\delta_{\beta}-1}{\delta_{\beta_1}} N_{\beta_1} N_{\beta_2} (\beta_1
\cdot \beta_2)\, . \label{boundary}
\end{equation}
Therefore, in order to compute $C_{\beta}$, it remains to evaluate the left--hand side of \eqref{Euler_du}.
It is easy to see that via the splitting principle,
\begin{align}
e(\xii^*_{\delta_{\beta}} \otimes \textnormal{ev}_{\delta_{\beta}}^* TX) & \,=\,
c_1(\xii^*_{\delta_{\beta}})^2 + c_1(\xii^*_{\delta_{\beta}})
\textnormal{ev}_{\delta_{\beta}}^*(x_1) + \textnormal{ev}_{\delta_{\beta}}^*(x_2).\label{Euler_splitting}
\end{align}

In Section \ref{ITC} we show that
\begin{align}
\langle \textnormal{ev}_{\delta_{\beta}}^*(x_2), ~[\overline{\mathcal{M}}]
\rangle &\,=\, x_2([X]) N_{\beta},
\label{c2} \\
\langle c_1(\xii^*_{\delta_{\beta}})\textnormal{ev}_{\delta_{\beta}}^*(x_1), ~[\overline{\mathcal{M}}] \rangle &\,=\,
-\frac{(x_1 \cdot x_1)}{(\beta\cdot x_1)} N_{\beta}  \nonumber  \\
& +\frac{1}{2 (\beta\cdot x_1)}\sum_{
\substack{\beta_1+ \beta_2\,=\, \beta, \\ \beta_1, \beta_2 \neq 0} }
\binom{\delta_{\beta}-1}{\delta_{\beta_1}} N_{\beta_1} N_{\beta_2} (\beta_1 \cdot \beta_2)(\beta_1 \cdot x_1)
(\beta_2 \cdot x_1), \label{c1L}\\
\langle c_1(\xii^*_{\delta_{\beta}})^2, ~[\overline{\mathcal{M}}] \rangle &\,= \,
-\frac{1}{2}\sum_{
\substack{\beta_1+ \beta_2= \beta, \\ \beta_1, \beta_2 \neq 0} }
\binom{\delta_{\beta}-1}{\delta_{\beta_1}} N_{\beta_1} N_{\beta_2} (\beta_1 \cdot \beta_2). \label{c1_sq}
\end{align}

Equations \eqref{c2}, \eqref{c1L} and \eqref{c1_sq} combined with \eqref{Euler_splitting}, \eqref{boundary}
and \eqref{Euler_du} yield \eqref{c_beta}.

\section{Intersection of Tautological Classes}\label{ITC}

We will now prove equations \eqref{c2}, \eqref{c1L} and \eqref{c1_sq}.

\begin{lmm}
\label{c1_divisor_ionel}
On $\overline{\mathcal{M}}$, the following equality of divisors holds:
\begin{align}
c_1(\xii_{\delta_{\beta}}^*) &\,=\, \frac{1}{(\beta \cdot x_1)^2}
\Big( (x_1 \cdot x_1) \mathcal{H} -2 (\beta \cdot x_1) \textnormal{ev}_{\delta_{\beta}}^*(x_1) +
\sum_{\substack{\beta_1+ \beta_2= \beta, \\ \beta_1, \beta_2 \neq 0}}
\mathcal{B}_{\beta_1, \beta_2} (\beta_2 \cdot x_1)^2 \Big),  \label{chern_class_divisor}
\end{align}
where $\mathcal{H}$ is the locus satisfying the extra condition that the curve
passes through a given point, $\mathcal{B}_{\beta_1, \beta_2}$ denotes the
boundary stratum corresponding to the splitting into a
degree $\beta_1$ curve and degree $\beta_2$ curve with the last marked point
lying on the degree $\beta_1$ component. 
\end{lmm}

\begin{proof}
The proof is similar to the one given in \cite{ionel_genus1}. Let
$\mu_1\, , \mu_2 \,\in\, X$ be two generic pseudocycles in $X$ that represent the class
$x_1$. Let $\widetilde{\mathcal{M}}$ be a cover of $\overline{\mathcal{M}}$ with two additional
marked points with the last two marked points lying on $\mu_1$ and $\mu_2$
respectively. More precisely,
\begin{align*}
\widetilde{\mathcal{M}} &\,:=\,  \textnormal{ev}_{\delta_{\beta} +1}^{-1}(\mu_1) \cap
\textnormal{ev}_{\delta_{\beta} +2}^{-1}(\mu_2) \subset
\overline{\mathcal{M}}_{0, \delta_{\beta}+2}(X, \beta)\, .
\end{align*}
Note that the projection $\pi\,:\, \widetilde{\mathcal{M}} \,\longrightarrow\,
\overline{\mathcal{M}}$ that forgets the last two marked points is a
$(\beta \cdot x_1)^2$--to--one map.

We now construct a meromorphic section
$$\phi\, :\,\widetilde{\mathcal{M}} \,\longrightarrow\, \xii_{\delta_{\beta}}^*$$
given by
\begin{align}
\phi ([u; y_{_1}, \cdots, y_{_{\delta_{\beta}-1}}; y_{_{\delta_{\beta}}};
y_{_{\delta_{\beta} +1}},
y_{_{\delta_{\beta}+2}}]) &:= \frac{(y_{_{\delta_{\beta}+1}} - y_{_{\delta_{\beta}+2}})
d y_{_{\delta_{\beta}}}}
{(y_{_{\delta_{\beta}}}-y_{_{\delta_{\beta}+1}})(y_{_{\delta_{\beta}}}-
y_{_{\delta_{\beta}+2}})}. \label{section_ionel}
\end{align}
The right--hand side of \eqref{section_ionel} involves an abuse of notation: it is to be
interpreted in an affine coordinate chart and then extended as a meromorphic section
on the whole of $\mathbb{P}^1$. Note that on $({\mathbb P}^1)^3$, the holomorphic
line bundle
$$
\eta\, :=\,
q^*_1K_{{\mathbb P}^1}\otimes{\mathcal O}_{({\mathbb P}^1)^3}(\Delta_{12}
+\Delta_{13}-\Delta_{23})$$ is trivial, where $q_1\, :\, ({\mathbb P}^1)^3\,
\longrightarrow\,{\mathbb P}^1$ is the projection to the first factor and
$\Delta_{jk}\,\subset\, ({\mathbb P}^1)^3$ is the divisor consisting of all points
$(z_i\, ,z_2\, ,z_3)$ such that $z_j\,=\, z_k$. The diagonal action of ${\rm PSL}(2,
{\mathbb C})$ on $({\mathbb P}^1)^3$ lifts to $\eta$ preserving its
trivialization. The section $\phi$ in \eqref{section_ionel} is given by this trivialization of $\eta$.

Since $c_1(\xii_{\delta_{\beta}}^*)$ is the zero divisor minus the
pole divisor of $\phi$, we gather that
\begin{align*}
c_1(\xii_{\delta_{\beta}}^*) &\,=\, \{ y_{_{\delta_{\beta}+1}} = y_{_{\delta_{\beta}+2}}\}
-\{y_{_{\delta_{\beta}}}=y_{_{\delta_{\beta}+1}} \}-
\{y_{_{\delta_{\beta}}}= y_{_{\delta_{\beta}+2}}\}\, .
\end{align*}
When projected down to $\overline{\mathcal{M}}$, the divisor
$\{ y_{_{\delta_{\beta}+1}} = y_{_{\delta_{\beta}+2}}\}$ becomes
$(x_1\cdot x_1)\mathcal{H} + (\beta_2 \cdot x_1)^2 \mathcal{B}_{\beta_1, \beta_2}$,
while both the divisors $\{y_{_{\delta_{\beta}}}=y_{_{\delta_{\beta}+1}} \}$ and
$\{y_{_{\delta_{\beta}}}=y_{_{\delta_{\beta}+2}} \}$ become
$(\beta \cdot x_1) \textnormal{ev}_{\delta_{\beta}}^*(x_1)$.
Since $\widetilde{\mathcal{M}}$
is a $(\beta \cdot x_1)^2$--to--one cover of $\overline{\mathcal{M}}$, we obtain \eqref{chern_class_divisor}.
\end{proof}

We are now ready to prove \eqref{c2}, \eqref{c1L} and \eqref{c1_sq}.

\begin{proof}[Proof of \eqref{c2}] Let $s: X \longrightarrow TX$ be a smooth section
transverse to the zero set. The number of points at which
it vanishes (counted with a sign) is $x_2([X])$.
Note that a section $s\,:\, X \,\longrightarrow\, TX$
induces a section $\textnormal{ev}_{\delta_{\beta}}^*s$ of the pullback
$ \textnormal{ev}_{\delta_{\beta}}^*TX \longrightarrow \overline{\mathcal{M}}$.
The zero set of $\textnormal{ev}_{\delta_{\beta}}^*s$ is a degree $\beta$
curve through the points $p_1\, ,\cdots\, , p_{\delta_{\beta-1}}$ and one of the
zeros of $s$.
Let us denote one of the zeros of $s$ to
be $q$.
Let $C_q$ be
one of the curves through
$p_1\, ,\cdots\, , p_{\delta_{\beta-1}}$ and $q$.
If $q$ is a positive zero of $s$, then $C_q$ is a positive zero  of
$\textnormal{ev}_{\delta_{\beta}}^*s$; the reverse is true if $q$ is a negative
zero of $s$. Hence, the number of zeros of $\textnormal{ev}_{\delta_{\beta}}^*s$
counted with a sign is $x_2([X]) N_{\beta}$, which proves \eqref{c2}.
\end{proof}

\begin{proof}[Proof of \eqref{c1L}] It is easy to see that
\begin{align}
\langle \textnormal{ev}_{\delta_{\beta}}^*(x_1) \mathcal{H}, ~[\overline{\mathcal{M}}] \rangle &=
(\beta \cdot x_1)N_{\beta}, \nonumber \\
\langle \textnormal{ev}_{\delta_{\beta}}^*(x_1)^2,
~[\overline{\mathcal{M}}] \rangle & = (x_1 \cdot x_1) N_{\beta} \qquad \textnormal{and} \nonumber \\
\sum_{\substack{\beta_1+ \beta_2= \beta, \\ \beta_1, \beta_2 \neq 0}}\langle
\textnormal{ev}_{\delta_{\beta}}^*(x_1) \mathcal{B}_{\beta_1, \beta_2}, ~[\overline{\mathcal{M}}] \rangle &=
\sum_{\substack{\beta_1+ \beta_2= \beta, \\ \beta_1, \beta_2 \neq 0}}
\binom{\delta_{\beta}-1}{\delta_{\beta_1}} N_{\beta_1} N_{\beta_2} (\beta_1 \cdot \beta_2)(\beta_1 \cdot x_1)
\nonumber \\
& = \frac{1}{2}\sum_{\substack{\beta_1+ \beta_2= \beta, \\ \beta_1, \beta_2 \neq 0}}
\binom{\delta_{\beta}-1}{\delta_{\beta_1}} N_{\beta_1} N_{\beta_2} (\beta_1 \cdot \beta_2) (\beta \cdot x_1).
\label{itc_easy}
\end{align}
Equations \eqref{itc_easy} and \eqref{chern_class_divisor} together imply \eqref{c1L}.
\end{proof}

\begin{proof}[Proof of \eqref{c1_sq}]
First of all, we note that
\begin{align}
\langle c_1(\xii_{\delta_{\beta}}^*) \mathcal{H},
~[\overline{\mathcal{M}}] \rangle &\,=\, -2 N_{\beta}\, . \label{itc_obvious}
\end{align}
Indeed, this follows immediately from \eqref{chern_class_divisor}.

We will now show that
\begin{align}
\langle c_1(\xii_{\delta_{\beta}}^*) \mathcal{B}_{\beta_1, \beta_2},
~[\overline{\mathcal{M}}] \rangle  &\,=\,
-\binom{\delta_{\beta}-1}{\delta_{\beta_1}} N_{\beta_1} N_{\beta_2}
(\beta_1 \cdot \beta_2)\, .\label{c1_on_boundary}
\end{align}
For this, let
$\mathcal{B}_{\beta_1, \beta_2}(p_{i_1}, \cdots, p_{i_{\delta_{\beta_1}}})$
denote the component of $\mathcal{B}_{\beta_1, \beta_2}$
where the $\beta_1$ curve passes through $p_{i_1}, \cdots, p_{i_{\delta_{\beta_1}}}$.
Now define the map
\begin{align}
\pi\,:\, \mathcal{B}_{\beta_1, \beta_2}(p_{i_1}, \cdots, p_{i_{\delta_{\beta_1}}})
&\,\longrightarrow\,
\overline{\mathcal{M}}_{0, \delta_{\beta_1}+1}(X, \beta_1; p_{i_1}, \cdots, p_{i_{\delta_{\beta_1}}})
\label{proj_map_not_quite}
\end{align}
which is the projection onto the $\beta_1$ component.
This map $\pi$ is $N_{\beta_2} (\beta_1 \cdot \beta_2)$--to--one.
Let \[ \widetilde{\xii}_{i_1, \cdots, i_{\delta_{\beta_1}}} \,\longrightarrow\,
\overline{\mathcal{M}}_{0, \delta_{\beta_1}+1}(X, \beta_1; p_{i_1}, \cdots, p_{i_{\delta_{\beta_1}}})\]
be the universal tangent bundle line at the last marked point.
By \eqref{itc_obvious},
\begin{align}
\langle c_1(\widetilde{\xii}^*_{i_1, \cdots, i_{\delta_{\beta_1}}}),
~[\overline{\mathcal{M}}_{0, \delta_{\beta_1}+1}(X, \beta_1; p_{i_1}, \cdots, p_{i_{\delta_{\beta_1}}})]
\rangle &\,=\, -2 N_{\beta_1}\, . \label{beta_replace_beta1}
\end{align}
Note that we replaced $\beta$ by $\beta_1$ in \eqref{itc_obvious} to obtain
the above equation;
that is permitted since \eqref{itc_obvious} holds for \textit{all} $\beta$.

Let
\[\{y \in \mathcal{G} \} \]
denote the divisor
inside the space $\mathcal{B}_{\beta_1, \beta_2}$ that corresponds to the marked point
lying on a ghost bubble. More precisely, the elements of
$\{y \in \mathcal{G} \} $ comprise of maps form a wedge of three spheres into $X$
that is degree $\beta_1$ on the first component, degree $\beta_2$ on the third component
and constant on the middle component, with the marked point lying on the
middle component.
As stated in \cite{ionel_genus1} (equation $2.10$, Page $29$), we have the following equality of divisors,
\begin{align*}
c_1(\widetilde{\xii}^*_{i_1, \cdots, i_{\delta_{\beta_1}}})\Big|_{\mathcal{B}_{d_1, d_2}} &=
\pi^* c_1(\widetilde{\xii}^*_{i_1, \cdots, i_{\delta_{\beta_1}}}) + |\{y \in \mathcal{G}\}|.
\end{align*}
Hence,
$$
\langle c_1(\xii_{\delta_{\beta}}^*) \mathcal{B}_{\beta_1, \beta_2},
~[\overline{\mathcal{M}}] \rangle
$$
$$
=\, \sum_{(i_1, \cdots, i_{\delta_{\beta_1}}) \subset
\{1, 2, \cdots, \delta_{\beta}-1 \}} \langle \pi^*c_1(\widetilde{\xii}^*_{i_1, \cdots, i_{\delta_{\beta_1}}}),
~[\overline{\mathcal{M}}_{0, \delta_{\beta_1}+1}(X, \beta_1; p_{i_1}, \cdots, p_{i_{\delta_{\beta_1}}})]
\rangle + \vert\{ y_{\delta_{\beta}} \in \mathcal{G}\}\vert
$$
$$
= \, -2N_{\beta_1} \binom{\delta_{\beta}-1}{\delta_{\beta_1}} N_{\delta_{\beta_2}}
(\beta_1 \cdot \beta_2) + \binom{\delta_{\beta}-1}{\delta_{\beta_1}}
N_{\delta_{\beta_1}} N_{\delta_{\beta_2}} (\beta_1 \cdot \beta_2)
$$
\begin{equation}\label{c1_intersect_bdry}
\,=\, -\binom{\delta_{\beta}-1}{\delta_{\beta_1}} N_{\beta_1} N_{\beta_2}
(\beta_1 \cdot \beta_2),
\end{equation}
which proves \eqref{c1_on_boundary}.
Equations \eqref{c1_on_boundary}, \eqref{itc_easy}, \eqref{c1L}
and \eqref{chern_class_divisor} imply that
\begin{align*}
\langle c_1(\xii^*_{\delta_{\beta}})^2, ~[\overline{\mathcal{M}}] \rangle & =
-\frac{1}{(\beta \cdot x_1)^2}\sum_{
\substack{\beta_1+ \beta_2= \beta, \\ \beta_1, \beta_2 \neq 0} }
\binom{\delta_{\beta}-1}{\delta_{\beta_1}} N_{\beta_1} N_{\beta_2} (\beta_1 \cdot \beta_2)
\Big((\beta_1 \cdot x_1) (\beta_2 \cdot x_1) + (\beta_2 \cdot x_1)^2 \big)
\\
& = -\frac{1}{2}\sum_{
\substack{\beta_1+ \beta_2= \beta, \\ \beta_1, \beta_2 \neq 0} }
\binom{\delta_{\beta}-1}{\delta_{\beta_1}} N_{\beta_1} N_{\beta_2} (\beta_1 \cdot \beta_2).
\end{align*}
This completes the proof.
\end{proof}

\section{Degenerate contribution to the Euler class}
\label{degenerate_contrib}

We start this section by recapitulating a few facts about moduli spaces of curves on
del-Pezzo surfaces. As before, let $X$
be a del-Pezzo surface and $\beta \in H_2(X, \mathbb{Z})$ a given homology class.
Since $X$ is algebraic, it embeds inside $\mathbb{P}^{n}$ for some $n$; fix such
an embedding. A map
$u\,:\, \mathbb{P}^1 \,\longrightarrow\, X$ also determines a map from $\mathbb{P}^1$
to $\mathbb{P}^n$ obtained by composing with the inclusion map. Let us say that a
degree $\beta$ map into $X$ determines a degree $d$ map into $\mathbb{P}^n$.
Given $\beta$, this $d$ is uniquely determined. We note that two distinct $\beta$ can
determine the same degree $d$. This yields an inclusion
\begin{equation}\label{li}
\mathcal{M}_{0,k}^{*}(X, \beta) \,\subset\, \mathcal{M}_{0,k}^*(\mathbb{P}^n, d)\, .
\end{equation}
It is well known that $\mathcal{M}_{0,k}^*(\mathbb{P}^n, d)$ is a smooth complex
manifold of the expected dimension. It is also known that
$\mathcal{M}_{0,k}^{*}(X, \beta)$ is a smooth manifold of the expected dimension. It also follows from the results of
\cite{D_Testa}
that $\overline{\mathcal{M}}_{0,0}(X, \beta)$ is an irreducible variety of the
expected  dimension.

For any $u\,:\, \mathbb{P}^1 \,\longrightarrow\, X$ representing a point of
$\mathcal{M}_{0,k}^{*}(X, \beta)$, let $\widehat{u}\, :\, \mathbb{P}^1
\,\longrightarrow\,\mathbb{P}^n$ be its composition with the embedding of $X$ in
$\mathbb{P}^n$. We have $T_u \mathcal{M}_{0,k}^{*}(X, \beta)\,=\, H^0(\mathbb{P}^1,\,
u^* N_{X/\mathbb{P}^1})$ and
$T_{\widehat u} \mathcal{M}_{0,k}^*(\mathbb{P}^n, d)\,=\, H^0(\mathbb{P}^1,\,
\widehat{u}^* N_{\mathbb{P}^n/\mathbb{P}^1})$, where $N_{X/\mathbb{P}^1}$ and
$N_{\mathbb{P}^n/\mathbb{P}^1}$ are the normal bundles. Since the homomorphism
$$
H^0(\mathbb{P}^1,\, u^* N_{X/\mathbb{P}^1})\, \longrightarrow\, H^0(\mathbb{P}^1,\,
\widehat{u}^* N_{\mathbb{P}^n/\mathbb{P}^1})
$$
induced by the differential of the embedding $X\, \hookrightarrow\,
\mathbb{P}^n$ is injective, the inclusion map in \eqref{li} is an immersion.
Hence, we conclude that $\mathcal{M}_{0,k}^{*}(X, \beta)$ is a submanifold
of $\mathcal{M}_{0,k}^*(\mathbb{P}^n, d)$.

We are now ready to state our neighborhood construction. Before that let us recapitulate a standard
of notation. We will be denoting an element of $\mathbb{P}^n$ as
\[ [Z_0, Z_1, \cdots, Z_n],\]
where $Z_i$ are not all zero. The square bracket $[ ~]$ is to remind us that we are
looking at an equivalence class of $n+1$ tuples. In other words, if $\lambda$ is a non-zero
complex number, then
\[ [Z_0, Z_1, \cdots, Z_n] = [ \lambda Z_0, \lambda Z_1, \cdots, \lambda Z_n] \in \mathbb{P}^n.\]

Let us also explain one terminology we will be using frequently. Suppose $Y$ is a
$k$-dimensional submanifold of an $m$-dimensional manifold $X$. Let $p$ be a point
in $Y$ and let $(x_1(p), x_2(p), \cdots, x_m(p))$ be a coordinate chart for $X$ around the
point $p$. Since $Y$ is a submanifold of $X$ there exist $i_1$, $i_2, \cdots, i_k$ such that
$x_{i_1}(p), \cdots, x_{i_k}(p)$ determines a coordinate chart for $Y$. We will call
these coordinates $x_{i_1}(p), \cdots, x_{i_k}(p)$ the \textit{free} coordinates. What
this means is the following: suppose $(x_1^t(p), \cdots ,x_m^t(p))$ is a point that lies in $Y$ and
is close to $(x_1(p), x_2(p), \cdots, x_m(p))$. If we know the coordinates $x_{i_1}(t),
\cdots, x_{i_k}(t)$ then we can solve for the remaining coordinates in terms of the
free coordinates. We will be using this observation quite frequently henceforth.

\subsection{Neighborhood Construction}

Let $v_{\A}\,:\, \mathbb{P}^1 \,\longrightarrow\, X$ and
$v_{\B}\,:\, \mathbb{P}^1 \,\longrightarrow\, X$ be two holomorphic
curves of degree $\beta_1$ and $\beta_2$ respectively, such that
\[ v_{\mathrm{A}}([1,0]) \,=\, v_{\mathrm{B}}([0,1])\]
(we consider $\mathbb{P}^1$ as equivalence classes of points of ${\mathbb C}^2\setminus
\{0\}$). Furthermore, assume that $v_{\A}$ and $v_{\B}$ are not
multiply covered.
Let $\beta\,:= \, \beta_1+ \beta_2$.
We will describe a procedure to
construct a degree $\beta$ curve that lies near
the degree $\beta$ bubble map determined by $v_{\A}$ and $v_{\B}$.

Since $X$ is projective, it embeds inside $\mathbb{P}^{n}$ for some $n$.
Suppose $v_{\A}$ and $v_{\B}$ are explicitly given as
\begin{align*}
v_{\A}([X,Y]) &\,:=\, [\A^0(X,Y), \cdots, \A^n(X,Y)] \qquad \textnormal{and} \qquad v_{\B}([X,Y]) :=
[\B^0(X,Y), \cdots, \B^n(X,Y)]
\end{align*}
where $\A^{\mu}(X,Y)$ and $\B^{\nu}(X,Y)$ are homogeneous polynomials of degrees $d_1$ and $d_2$.
Let $\A^{\mu}_{i}$ and $\B^{\nu}_{j}$ be the coefficient of $X^i$ in
$\A^{\mu}(X,Y)$ and $\B^{\nu}(X,Y)$ respectively.
By composing with  appropriate M\"obius transformations
(that fix $[1,0]\,\in\, \mathbb{P}^1$ and $[0,1]\,\in\, \mathbb{P}^1$ respectively), we
can set three of the $\A^{\mu}_{i} $ and three of the $\B^{\nu}_{j}$ to be some
specific constant. Now define
\begin{align*}
\A^{\mu}_{t}(X,Y) &\,:=\, \sum_{i=0}^{d_1} (\A^{\mu}_i + t^{\mu}_i) X^i Y^{d_1-i} \qquad \textnormal{and}
\qquad \B^{\nu}_s(X,Y) \,:= \,\sum_{j=0}^{d_2} (\B^{\nu}_j + s^{\nu}_j) X^j Y^{d_2-j}\, ,
\end{align*}
where $t^{\mu}_i$ and $s^{\nu}_j$ are small complex numbers. 
After composing with an automorphism of $\mathbb{P}^n$ if necessary, we may 
assume that $\A^{\mu}(0,1), \B^{\mu} (1,0) \neq 0 \ \forall \ \mu$.
Set the three of the
$t^{\mu}_i$ and $s^{\nu}_j$ to be zero
(the ones that correspond to the six coefficients
that were fixed).
Next, given an $\varepsilon$ that is small, define
\begin{align*}
\R^{\mu}_{\varepsilon, t,s} (X,Y) &\,:=\, \A^{\mu}_t (X,Y) Y^{d_2} +
\frac{\A^{\mu}_t(1,0)}{\B_s^{\mu}(0,1)} \B^{\mu}_s(\varepsilon^2 X, Y) X^{d_1}
-\A_t^{\mu}(1,0) X^{d_1} Y^{d_2}\, , \\
\J^{\mu}_{\varepsilon, t,s} (X,Y) &\,:=\, \frac{\B_s^{\mu}(0,1)\R^{\mu}_{\varepsilon, t,s}
(X, \varepsilon^2 Y)}{\A_t ^{\mu} (1,0)\varepsilon^{2 d_2}}\, .
\end{align*}

If the polynomials
$\R^{\mu}_{\varepsilon, t,s}$ induce a well-defined map into $\mathbb{P}^n$, then denote
$u_{\varepsilon, t, s}^{\R}\,:\,\mathbb{P}^1\,\longrightarrow\, \mathbb{P}^n$
to be the
degree $d\,:=\, d_1+d_2$ map defined by these polynomials.
Note that for generic $(\varepsilon, t,s)$, the polynomials $\R^{\mu}_{\varepsilon, t,s}$
do induce a well-defined map (i.e., all the coordinates are not zero).
Next, we observe that $u_{\varepsilon, t, s}^{\R}$ does not necessarily map into $X$.
In order for the curve to lie in $X$, only $\delta_{\beta_1}+1$ of the
$t^{\mu}_i$ and $\delta_{\beta_2}+1$ of the $s^{\nu}_j$ are free. Denote the
free variables by $t^{\mu}_i$ and  $s^{\nu}_j$ and the remaining ones by
$\widehat{t}^{\mu}_i$ and $\widehat{s}^{\nu}_j$ respectively.
Let us denote the corresponding
polynomials to be $\widehat{\R}$ and $\widehat{\J}$ respectively, and
let
$$u_{\varepsilon, t, s}^{\widehat{\R}}\,:\,\mathbb{P}^1
\,\longrightarrow\, \mathbb{P}^n$$
be the corresponding degree $d$ map.
By definition, now $u_{\varepsilon, t, s}^{\widehat{\R}}$ lies in $X$.

Next, let
\begin{align*}
\{p_i\,:=\, [a^0_i, a^1_i,\cdots, a^n_i]\}_{i=1}^{\delta_{\beta_1}}
\bigcup  \{q_j:= [b_j^0, b_1^2, \cdots,  b_j^n]\}_{j=1}^{\delta_{\beta_2}}
\end{align*}
be a collection of $\delta_{\beta}-1$ generic points in $X$.
Furthermore, let
\begin{align*}
\{\mkdpt_i:= [x^a_i, y^a_i]\}_{i=1}^{\delta_{\beta_1}}
\bigcup \{\mkdptt_j:= [x^b_j, y^b_j]\}_{j=1}^{\delta_{\beta_2}}
\end{align*}
be a  collection of $\delta_{\beta}-1$ points in $\mathbb{P}^1$
such that
\begin{align*}
v_{\A}([x^a_i, y^a_i]) &\,=\, [a^0_i, a^1_i, \cdots, a^n_i] \qquad \forall ~~i
\,=\, 1, \cdots ,\delta_{\beta_1} \qquad \textnormal{and}\\
v_{\B}([x^b_j, y^b_j]) &\,=\, [b^0_j, b^1_j, \cdots, b^n_j] \qquad \forall ~~j
\,=\, 1,\cdots ,\delta_{\beta_2}.
\end{align*}
This gives us a set of $2\delta_{\beta}$ equations
\begin{align}
\label{equation_explicit_transverse}
\A^{\mu}(x^a_i, y^a_i) a^0_i - a^{\mu}_i \A^0(x^a_i, y^a_i)&=0,
\qquad \forall  ~~\mu = \mu_1, ~\mu_2 ~~\textnormal{and} ~~i
\,=\, 1,\cdots , \delta_{\beta_1}, \nonumber \\
\B^{\nu}(x^b_j, y^b_j) b^0_j - b^{\nu}_j \B^0(x^b_j, y^b_j)&=0,
\qquad \forall  ~~\nu = \nu_1, ~\nu_2 ~~\textnormal{and} ~~j = 1,
\cdots , \delta_{\beta_2}, \nonumber \\
\A^{\mu}(1,0)\B^0(0,1)-\B^{\mu}(0,1)\A^0(1,0) &=0, \qquad \forall  ~~\mu = \widetilde{\mu}_1,
~\widetilde{\mu}_2,
\end{align}
for some $\mu_i, \nu_i$ and $\widetilde{\mu}_i \,\in\, \{0,1, 2, \cdots, n\}.$
Without loss of generality, we set
$y^a_i \,=\,1$ for all $i \,=\, 1, \cdots ,\delta_{\beta_1}$ and
set $x^b_j \,=\,1$ for all $j\,=\,1, \cdots ,\delta_{\beta_2}$.
Since $\delta_{\beta_1}+1$ of the
$\A^{\mu}_i$ are free and $\delta_{\beta_2}+1$ of the
$\B^{\nu}_j$ are free, it follows that
the number of the free unknowns $\A^{\mu}_i$, $\B^{\nu}_j$, $x^a_i$ and $y^b_j$
is $2\delta_{\beta}$. Note that the evaluation map
\begin{align*}
\textnormal{ev}\,:\, \mathcal{M}_{0, \delta_{\beta_1}}^*(X, \beta_1) \times
\mathcal{M}_{0, \delta_{\beta_2}}^*(X, \beta_2) &
\,\longrightarrow\, X^{\delta_{\beta_1}-1} \times X^{\delta_{\beta_2}-1} \times X^2.
\end{align*}
is transverse to $(p_1, \cdots, p_{\delta_{\beta_1}-1}, q_1, \cdots,
q_{\delta_{\beta_2}-1} ) \times \Delta_{X}$ if the points $(p_1, \cdots,
p_{\delta_{\beta_1}-1}, q_1, \cdots, q_{\delta_{\beta_2}-1} )$ are generic. In other
words, for a generic choice of these $\delta_{\beta}-1$ points, the equations
\eqref{equation_explicit_transverse} simultaneously vanish transversely at the given value.

\begin{lmm}
\label{ift_gluing}
Let $\varepsilon$ be a given small nonzero complex number. Then there exists a unique
triple
$(t(\varepsilon), s(\varepsilon), \theta(\varepsilon))$ that is small and solves the
following set of equations:
\begin{align}
a^0_i \widehat{\R}^{\mu}_{\varepsilon, t,s}(x^a_i+ \theta^a_i, 1) -
a^{\mu}_i \widehat{\R}^{0}_{\varepsilon, t,s}(x^a_i+ \theta^a_i, 1)  &\,=\, 0,
\qquad \forall  ~~\mu = \mu_1, ~\mu_2 ~~\textnormal{and} ~~i \,=\, 1,\cdots , \delta_{\beta_1},
\nonumber \\
b^0_i \widehat{\J}^{\nu}_{\varepsilon, t,s}(1, y^b_j+ \theta^b_j) -
b^{\nu}_i \widehat{\J}^{0}_{\varepsilon, t,s}(1, y^b_j)  &\,=\, 0,
\qquad \forall  ~~\nu = \nu_1, ~\nu_2 ~~\textnormal{and} ~~j \,=\, 1,\cdots ,\delta_{\beta_2},
\nonumber \\
\A^{\mu}_t(1,0)\B^0_s(0,1)-\B^{\mu}_s(0,1)\A^1_t(1,0) &\,=\,0, \qquad \forall  ~~\mu = \widetilde{\mu}_1,
~\widetilde{\mu}_2. \label{equation_perturbed_transverse}
\end{align}
\end{lmm}

\begin{proof}
The equation \eqref{equation_explicit_transverse}
has a transverse solution while equation \eqref{equation_perturbed_transverse}
has approximately the same linearization as \eqref{equation_explicit_transverse}
(i.e., to zero-th order
in $\varepsilon$, the linearizations are the same).  The lemma now follows from the
implicit function theorem.
\end{proof}

\begin{cor}\label{injective_transverse_intersection}
Given a nonzero $\varepsilon$ that is small, let $t(\varepsilon)$,
$s(\varepsilon)$ and $\theta(\varepsilon)$ be the unique solution as given by
Lemma \ref{ift_gluing}. If $\varepsilon_1\,\neq\, \varepsilon_2$ then the curves
$u_{\varepsilon_i, t(\varepsilon_1), s(\varepsilon_i)}^{\widehat{\mathcal{R}}}$
($i\,=\,1\, , 2$)
both pass through $p_1, \cdots, p_{\delta_{\beta_1}-1}$, $q_1, \cdots,
q_{\delta_{\beta_2}-1}$
and they intersect transversally. In particular, the two curves are distinct.
\end{cor}

\begin{proof}
Follows immediately from Lemma \ref{ift_gluing}.
\end{proof}

\begin{rem}
Note that if the $\delta_{\beta}-1$ points are generic, the map
$u_{\varepsilon, t(\varepsilon), s(\varepsilon)}^{\widehat{\mathcal{R}}}$ is
well-defined and not multiply covered.
\end{rem}

\begin{lmm}
\label{gluing_surjective}
Given a nonzero $\varepsilon$ that is small, let $t(\varepsilon)$,
$s(\varepsilon)$ and $\theta(\varepsilon)$ be the
unique solution as given by Lemma \ref{ift_gluing},
and let $\varphi_{\varepsilon}$ be the M\"obius transformation
given by
\begin{align*}
\varphi_{\varepsilon}([X,Y]) \, :=\, [X, \varepsilon Y]\, .
\end{align*}
Let $m$ be some fixed complex number different from $0$ and let
$\fmkdpt_z\,:=\, [1, m+z]\,\in\, \mathbb{P}^1$.\footnote{The reader can set $m:=1$, but it is
instructive to point out that for our arguments to hold
it doesn't matter what $m$ is as long as it is not $0$. Basically we want the point $\fmkdpt_z$ to be
away from $[1,0]$ and $[0,1]$.}
Let $v$ be a bubble map with $\delta_{\beta}$ marked points that is
defined as follows: it is determined by the maps
$v_{\A}$ and $v_{\B}$; the first
$\delta_{\beta_1}$ marked points lie on the $\A$-bubble and are required to pass through
the points $p_{1}$ to $p_{\delta_{\beta_1}}$;
the next $\delta_{\beta_2}$
marked points lie on the $\B$-bubble and are required to pass
through $q_1$ to $q_{\delta_{\beta_2}}$
and the last marked point $\fmkdpt$ (which is free) lies on a ghost bubble connecting
the $\A$-bubble and the $\B$-bubble.
Let $\mathrm{B}_{\mathrm{r}}(0)$ be an open ball of radius $\mathrm{r}$
around the origin in $\mathbb{C}$. Then, the map
\begin{align*}
\Phi & :(\mathrm{B}_{\mathrm{r}}(0)-0) \times \mathrm{B}_{\mathrm{r}}(0)  \longrightarrow\,
\mathcal{M}_{0, \delta_{\beta}}^*
(X, \beta; p_1, \cdots, p_{\delta_{\beta_1}}, q_1, \cdots,
q_{\delta_{\beta_2}})
\end{align*}
given by
\begin{align*}
\Phi(\varepsilon, z)~ :=
\end{align*}
\begin{align*}
[u^{\widehat{\R}}_{\varepsilon, t(\varepsilon), s(\varepsilon) }
\circ \varphi_{\varepsilon}; ~\varphi_{\varepsilon}^{-1}\circ\mkdpt_1(\theta(\varepsilon)), \cdots,
\varphi_{\varepsilon}^{-1}\circ \mkdpt_{\delta_{\beta_1}-1}(\theta(\varepsilon));
\varphi_{\varepsilon}^{-1}\circ \mkdptt_1(\theta(\varepsilon)), \cdots, \varphi_{\varepsilon}^{-1}\circ
\mkdptt_{\delta_{\beta_2}-1}(\theta(\varepsilon)); \fmkdpt_z]
\end{align*}
is a bijective map onto an open neighborhood of $v$.
\end{lmm}

\begin{proof}  First let us show that
$\Phi(\varepsilon, z)$ extends continuously to $\{0\}\times \B_{\mathrm{r}} (0)$ in the Gromov topology. Let $(\varepsilon_n, z_n)$
be a sequence that converges to $(0,0)$.
By Theorem $4.6.1$ and Definition $5.2.1$ in \cite{McSa},
we conclude that $\Phi(\varepsilon_n, z_n)$ converges in Gromov topology to $[v]$. Elsewhere $\Phi$ is obviously continuous.

Next, we observe that
Corollary \ref{injective_transverse_intersection} combined with the fact that $\alpha_z$ is
first order in $z$, implies that
$\Phi(\varepsilon, z)$ is injective.
It remains to show that $\Phi(\varepsilon, z)$ surjects onto an open neighborhood of $[v]$.
First, define
\begin{align*}
\pi\,:\, \overline{\mathcal{M}}_{0, \delta_{\beta}}
(X, \beta; p_1, \cdots, p_{\delta_{\beta_1}}, q_1, \cdots,
q_{\delta_{\beta_2}})\,\longrightarrow\, \overline{\mathcal{M}}_{0, \delta_{\beta}-1}
(X, \beta; p_1, \cdots, p_{\delta_{\beta_1}}, q_1, \cdots,
q_{\delta_{\beta_2}})
\end{align*}
to be the map that forgets the last marked point. Also define
the following map
\begin{align*}
\Psi \,:\,(\mathrm{B}_{\mathrm{r}}(0)-0)  \,\longrightarrow\,
\mathcal{M}_{0, \delta_{\beta}-1}^*
(X, \beta; p_1, \cdots, p_{\delta_{\beta_1}}, q_1, \cdots,
q_{\delta_{\beta_2}})\, , \ \
\varepsilon\,\longmapsto \,\pi \circ \Phi (\varepsilon, z).
\end{align*}
Note that $\pi ([v])$ is the bubble map determined by $v_{\A}$
and $v_{\B}$, with the $\delta_{\beta}-1$ marked points
distributed accordingly on the domain, but with no free marked point. In
particular there is no ghost bubble. Now, by Theorem $10.1.2$
and the arguments presented on Page $384$ and $385$ in \cite{McSa},
we conclude that there exists a smooth surjection from an open
ball in $\mathbb{R}^2$ to an open neighborhood of $\pi ([v])$ in
$$\mathcal{M}_{0, \delta_{\beta}-1}^*
(X, \beta; p_1, \cdots, p_{\delta_{\beta_1}}, q_1, \cdots,
q_{\delta_{\beta_2}}).$$
Hence an open neighborhood of $\pi([v])$ is one copy of $\mathbb{C}$
(i.e., it has just one branch). Since $\Psi$ (like $\Phi$) also extends as a continuous injection, it has to be a surjection onto an open neighborhood of $0$ by invariance of domain. Finally, we need to show that $\Phi$ is surjective. First we
note that if $[\widetilde{u}_{\varepsilon}, \widetilde{\mkdpt}(\varepsilon), \widetilde{\mkdptt}(\varepsilon),
\widetilde{\alpha}_z] $
Gromov converges to $[v]$ then $[\widetilde{u}_{\varepsilon}, \widetilde{\mkdpt}(\varepsilon),
\widetilde{\mkdptt}(\varepsilon)]$
Gromov Converges to $\pi([v])$. By the surjectivity of $\Psi$, we conclude that
$$[\widetilde{u}_{\varepsilon}, \widetilde{\mkdpt}(\varepsilon),
\widetilde{\mkdptt}(\varepsilon)] \,=\, \Psi(\varepsilon^{\prime})$$
for some $\varepsilon^{\prime}$. Finally, we observe that if
$[\Psi(\varepsilon^{\prime}), \widetilde{\alpha}_{z}]$ converges to
$[v]$ then $\widetilde{\alpha}_z \,=\, [1,m+z]$.
Hence, $\Phi$ is a surjection onto an open neighborhood of $[v]$.
\end{proof}

\begin{rem}
\label{mulitply_covered_curves}
In the above proof some care is needed to use Theorem $10.1.2$ in \cite{McSa}.
In \cite{McSa}, the Gluing Theorem (i.e., Theorem $10.1.2$)
holds when we are allowed to vary almost complex structure; however
the authors explicitly state that their Theorems do hold for a fixed (almost) complex
structure \textsf{provided we restrict ourselves to non multiply covered curves}. Hence,
for our arguments to go through, it is essential that the maps $v_{\A}$ and $v_{\B}$
are non-multiply covered. In fact, if $v_{\A}$ or $v_{\B}$ were multiply
covered, then our assertion that the normal neighborhood has a single branch is
false. There are examples of multiply
covered curves in the boundary of $\overline{\mathcal{M}}_{0,0}(X, \beta)$ whose
normal neighborhood has more than one branch when $X$ is a del-Pezzo surface.
However,
since the $\delta_{\beta}-1$ points are generic,
multiply covered curves
do not arise and hence do not play a role in our computations.
\end{rem}

\subsection{Multiplicity Computation}
\label{multiplicity_proof}

We are now ready to compute the multiplicity.
Since the map $\Phi$ defined in Lemma \ref{gluing_surjective}
is a bijection, it suffices to count the number of
solutions to the set of equations
\begin{align*}
d u_{\varepsilon}^{\widehat{\mathcal{R}}}\circ \varphi_{\varepsilon}\vert_{\fmkdpt_z}
\,= \,\nu
\end{align*}
for a small perturbation $\nu$. Here $u_{\varepsilon, t(\varepsilon),
s(\varepsilon)}^{\widehat{\mathcal{R}}}$ is abbreviated as
$u_{\varepsilon}^{\widehat{\mathcal{R}}}$.
Write this equation in local coordinates. Set
$w\,:= \,\frac{Y}{X}$. Assuming that the zeroth and the first coordinates in
$\mathbb{P}^n$ are free and the $n$-th coordinate
is non zero, the map $u_{\varepsilon}^{\widehat{\mathcal{R}}}\circ \varphi_{\varepsilon}$ in
local coordinates is
given by $\F_{0}(\varepsilon, w), \F_1(\varepsilon, w)$ where $\F_{\mu}$ are
defined to be
\begin{align*}
\F_{\mu}(\varepsilon, w) &\,:=\, \frac{\A^{\mu}(1, \varepsilon w)  +
\frac{\A^{\mu}(1,0)}{\B^{\mu}(0,1)} \frac{\B^{\mu}(\varepsilon, w)}{w^{d_2}} - \A^{\mu}(1,0)}
{\A^{n}(1, \varepsilon w)  +
\frac{\A^n(1,0)}{\B^n(0,1)} \frac{\B^n(\varepsilon, w)}{w^{d_2}} - \A^{n}(1,0)}, \qquad \mu = 0, 1.
\end{align*}
Let
\begin{align*}
\G_{\mu}(\varepsilon, w) &\,:=\, \frac{\partial \F_{\mu}(\varepsilon, w)}{ \partial w}, \qquad \mu =0, 1.
\end{align*}
It is now easy to see that
\begin{align}
\G_{\mu} &\,:=\, -\frac{\A^{\mu}_0}{\A^n_0}
\Big(\frac{1}{\B^{\mu}_{d_2}} - \frac{\A^n_1}{\A^n_0 \B^n_{d_2}} \Big)
\varepsilon + \mathrm{H}(\varepsilon, w) \varepsilon^2, \qquad
\mu =0, 1, \label{g_mu}
\end{align}
where $\mathrm{H}(\varepsilon, w)$ is holomorphic function near $(0,m)$. (Recall
that we have taken $m$ to be a fixed nonzero complex number.)
Since $\G_{\mu}(\varepsilon, w)$ is well-defined it follows that
$\A^n_{0}$ and $\B^{n}_{d_2}$ are non-zero.
Note that $\G_0(0,w) \,=\,0$ and $\G_1(0,w)\,=\,0$. We need to find the order to which
this vanishing takes place near the point $(0, m)$.
We now make a change of coordinates near $(0,m)$ given
by
\begin{align}
\widetilde{\varepsilon} &:= \G_0(\varepsilon, w) \qquad \textnormal{and} \qquad \widetilde{w} := w.
\label{coord_change}
\end{align}
If
\begin{align}
-\frac{\A^{0}_0}{\A^n_0}
\Big(\frac{1}{\B^{0}_{d_2}} - \frac{\A^n_1}{\A^n_0 \B^n_{d_2}} \Big) &\neq 0, \label{change_coord_generic}
\end{align}
then \eqref{coord_change} defines a change of coordinates near $(0,m)$ (see \eqref{g_mu}).
Let us justify why we can assume \eqref{change_coord_generic} holds. First of all we note
that
\[ v_{\A}([0,1]) = [\A^{0}_0, \A^1_0, \A^2_0, \cdots, \A^n_0] \qquad \textnormal{and}
\qquad v_{\B}([1,0]) = [\B^{0}_{d_2}, \B^1_{d_2}, \B^2_{d_2}, \cdots, \B^n_{d_2}].\]
Consider the following $\mathbb{C}^2$ inside $\mathbb{C}^{n+1}$, namely: 
\[ L := (*, *, 0 , \cdots, 0) \subset \mathbb{C}^{n+1}. \]
We will now consider automorphisms of $\mathbb{P}^n$ induced from an 
automorphism of $\mathbb{C}^{n+1}$ that acts non trivially on $L$ 
and acts as identity on 
\[ (0,0, *, \cdots, *) \subset \mathbb{C}^{n+1}. \]
We claim that we 
can find such an automorphism 
moving $v_{\A}([0,1])$ and $v_{\B}([1,0])$ to
two points such that
\begin{align}
\A^{0}_0, \A^1_0 \neq 0 \qquad \textnormal{and} \qquad
\Big(\frac{1}{\B^{0}_{d_2}} - \frac{\A^n_1}{\A^n_0 \B^n_{d_2}} \Big),
~\Big(\frac{1}{\B^{1}_{d_2}} - \frac{\A^n_1}{\A^n_0 \B^n_{d_2}} \Big) \neq 0. \label{aut_eqn}
\end{align}
To see why this is so, we will consider three cases. 
Suppose $\B^0_{d_2} \neq \frac{\A^n_0 \B^n_{d_2}}{\A^n_1}$ 
and $\B^1_{d_2} \neq \frac{\A^n_0 \B^n_{d_2}}{\A^n_1}$. Then 
we take an automorphism that fixes $\B^{0}_{d_2}$ and $\B^{1}_{d_2}$ 
and takes both $\A^0_0$ and $\A^1_0$ to something non zero. 
Next, if $\B^0_{d_2} = \frac{\A^n_0 \B^n_{d_2}}{\A^n_1}$, 
but $\B^1_{d_2} \neq \frac{\A^n_0 \B^n_{d_2}}{\A^n_1}$, 
then we take an automorphism of $\mathbb{C}^{n+1}$ 
that takes $\B^0_{d_2}$ to $2 \B^0_{d_2}$, takes $\B^1_{d_2}$ 
to $\B^1_{d_2}$ and takes both $\A^0_0$ and $\A^1_0$ to something non zero. 
Similar argument holds if $\B^0_{d_2} \neq  \frac{\A^n_0 \B^n_{d_2}}{\A^n_1}$, 
but $\B^1_{d_2} =0 \frac{\A^n_0 \B^n_{d_2}}{\A^n_1}$. Finally, suppose 
$\B^0_{d_2} = \frac{\A^n_0 \B^n_{d_2}}{\A^n_1}$ and 
$\B^1_{d_2} =0 \frac{\A^n_0 \B^n_{d_2}}{\A^n_1}$. Then 
we take an automorphism of $\mathbb{C}^{n+1}$ that takes 
$\B^0_{d_2}$ and $\B^1_{d_2} $ to $2\B^0_{d_2}$ and $2\B^1_{d_2}$ 
and both $\A^0_0$ and $\A^1_0$ to something non zero. That covers 
all the cases. 

Equation \eqref{aut_eqn} implies that \eqref{change_coord_generic} holds; in addition, it
also implies  that
\begin{align}
-\frac{\A^{0}_0}{\A^n_0}
\Big(\frac{1}{\B^{1}_{d_2}} - \frac{\A^n_1}{\A^n_0 \B^n_{d_2}} \Big) &\neq 0, \label{change_coord_generic_again}
\end{align}
holds. Note that since our automorphism only acts on $L$, the initial assumptions we made about
the zeroth and first coordinate being free and the $n^{\textnormal{th}}$ coordinate being
non zero, is still valid.
Let
\begin{align*}
\widetilde{\G}_{\mu}(\widetilde{\varepsilon}, \widetilde{w}) &:= \G_{\mu}(\varepsilon, w), \qquad \mu =0,1.
\end{align*}
Hence
\begin{align*}
\G_0(\widetilde{\varepsilon}, \widetilde{w}) &:= \widetilde{\varepsilon} \qquad \textnormal{and}
\qquad \G_1(\widetilde{\varepsilon}, \widetilde{w}) := \widetilde{\varepsilon}
\mathrm{K}(\widetilde{\varepsilon}, \widetilde{w}),
\end{align*}
for some function $\mathrm{K}(\widetilde{\varepsilon}, \widetilde{w})$.
It is easy to see that since \eqref{change_coord_generic} and \eqref{change_coord_generic_again} hold,
$\mathrm{K}(0, m) \neq 0$ if $m \neq 0$.
Next, let $\nu_0(\widetilde{\varepsilon}, w)$
and $\nu_1(\widetilde{\varepsilon}, w)$ be two holomorphic
perturbations (they are defined only in a neighborhood of $(0,m)$),
whose Taylor
expansion near $(0,m)$ is given by
\begin{align}
\nu_0(\widetilde{\varepsilon}, w) &:= a_{00} + a_{10} \widetilde{\varepsilon} + a_{01} (w-m) + \ldots
\qquad \textnormal{and} \nonumber \\
\nu_1(\widetilde{\varepsilon}, w) & := b_{00} + b_{10} \widetilde{\varepsilon} + b_{01} (w-m) + \ldots
\label{nu_taylor}
\end{align}
By saying that $\nu_{0}$ and $\nu_1$ are perturbations, we mean that the
the constant terms $a_{00}$ and $b_{00}$ in
the Taylor expansion are all small.
Now, we need to solve for
\begin{align}
\widetilde{\varepsilon} & =  \nu_0(\widetilde{\varepsilon}, w) \qquad
\textnormal{and} \label{nu0}\\
\widetilde{\varepsilon}\mathrm{K}(\widetilde{\varepsilon}, \widetilde{w})
& = \nu_1(\widetilde{\varepsilon}, w). \label{nu1}
\end{align}
Using \eqref{nu1} and \eqref{nu_taylor},
we conclude that
\begin{align}
\widetilde{\varepsilon} &= \frac{b_{00}}{\mathrm{K}(0,m)} + \mathrm{O}(w-m). \label{ep_tilde_zeroth_order}
\end{align}
Using \eqref{ep_tilde_zeroth_order}, \eqref{nu0}
and \eqref{nu_taylor}, we conclude that
\begin{align}
w-m & =\frac{\frac{b_{00}(1-a_{10})}{\mathrm{K}(0,m)} - a_{00}}{a_{10}} + \mathrm{O}((w-m)^2). \label{w_minus_m_soln}
\end{align}
Equation \eqref{w_minus_m_soln} implies that
if the perturbation $\nu$ is generic then
there exists a unique solution $w$ in a sufficiently small neighborhood
of $m$. Since the $\nu_i$ are chosen to be holomorphic, it will be counted with
a positive sign. Hence the multiplicity is one.

\section{Proof of transversality and general position arguments}
\label{transverse_proof}

It remains to show that the section $\psi$ induced by taking the derivative at a
marked point is transverse to the zero set, and that there exists a rational curve in $(X,\beta)$ that has exactly one genuine cusp and is an immersion otherwise. 
We prove the latter first. In what follows whenever we count nodes we do so along with their multiplicities.  

\begin{prp}\label{Existenceofagenuinecusp}
Let $X$ be the blowup of $\mathbb{P}^2$ at $k$ generic points $q_i$ (with exceptional divisors $E_i$) 
and let $\beta := d L - \displaystyle \sum _i m_i E_i$ 
be a homology class such that $N_{\beta-3L} >0$. 
Then, there exists a non-multiply covered rational curve in 
the class $\beta$ 
having exactly one genuine cusp; furthermore, the curve is an immersion everywhere else.  
\end{prp}

\begin{proof} Since $N_{\beta-3L}>0$, there exists an immersion 
$v: \mathbb{P}^1 \longrightarrow X$ representing the class $\beta-3L$ 
(cf. Theorem $4.1$, \cite{Pandh_Gott}). Let 
$c$ represent the homology class of a 
genuine cuspidal cubic in $X$; i.e. the homology class of $c$ is $3L$.   
Choose the cubic to intersect 
$v$ transversally at $3(d-3)$ 
non-singular points. 
We will construct a cuspidal curve in the class 
$(\beta -3L) + 3L$ by 
considering the bubble map formed by 
$v_A = v$ and $v_B = c$ where we may assume 
without loss of generality that $v_A ([1:0])=v_B([0:1])$. 
Now we are in the setting of Section \ref{degenerate_contrib}. 
Using the same notation as in that section, consider the map 
$u_{\varepsilon, t, s}^{\widehat{\R}}$ (whose image is by construction is required to lie in $X$). 
If $\varepsilon$ is small enough 
then it is easy to see that the number of nodes of 
this perturbed degree $d$ curve is $3(d-3)-1$ more than the number for 
$v_A$ (which is incidentally, ${d-4 \choose 2}-t$ where 
$t$ depends only on $m_i$) because we ``resolve" a node corresponding to the 
bubble point (which is the 
point $v_A([1:0])= v_B([0,1])$).

Let $v_B$ have its cusp at the point $[1:z]$. 
If we require $u_{\varepsilon, t, s}^{\widehat{\J}}$ to have a cusp at $[1:y]$ then this 
implies (assuming that the coordinate $\frac{X^{\mu_1}}{X^0}$ and $\frac{X^{\mu_2}}{X^0}$ are free in $X$) that 
\begin{gather}
\left(d\left(\frac{u^{\mu_1}}{u^{0}}\right), d\left(\frac{u^{\mu_2}}{u^{0}}\right)\right) = (0,0) \nonumber \\
\Rightarrow \left(\frac{Q_1 (y)}{Q_0(y)}, \frac{Q_2(y)}{Q_0(y)}\right) = (0,0),
\end{gather}
where $Q_i$ are polynomials whose coefficients depend rationally on $\varepsilon, t, s$ and $Q_1, Q_2$ 
have a common root $z$ when $\varepsilon=0=s$. Since the variety defined by $Q_1, Q_2$ (treated as a 
function of $\varepsilon, y, s$) is non-empty near $0,z,0$, there exists a small 
enough $(y-z,\varepsilon,s)$ such that $\varepsilon \neq 0$ so that $du(y)=0$.

In summary, the perturbed map has one cusp and $3(d-3)-1+{d-4\choose 2}-t$ nodes. 
By genus considerations the maximum number of nodes/cusps that a 
rational curve of class $\beta$ have is ${d-1\choose 2}-t$. 
Since that maximum number has been attained, 
the cusp has to be a genuine one (since anything worse than a cusp would contribute more than $1$ to the genus).
\end{proof}

Using a similar construction as in Proposition \ref{Existenceofagenuinecusp} we prove 
that cuspidal curves form a submanifold by constructing a genuine cuspidal curve at which transversality holds. 

\begin{lmm}\label{Transversality}
Let 
$\psi\,:\, \mathcal{M}_{0,\delta _{\beta}-1}^{*}(X, \beta ; p_1, p_2,\ldots,p_{\delta _{\beta}-1}) \,\longrightarrow\,
\xii_1^* \otimes \textnormal{ev}_1^* TX$
be the section induced by taking the derivative at the marked point, i.e.,
\[ \psi([u; z])\,:=\,  du\vert_{z}\, . \]
If $N_{\beta -3L} >0$, then $\psi$ is transverse to the zero section
\end{lmm}
\begin{proof}
We will actually construct a (genuine) cuspidal curve in $(X,\beta)$ such that 
the section $\psi\,:\, \mathcal{M}_{0,1}^{*}(X, \beta) \,\longrightarrow\,
\xii_1^* \otimes \textnormal{ev}_1^* TX$ (i.e., $X$ with one marked point) is transverse at this curve. This means that curves at which transversality fails form a proper subvariety and therefore, by the requirement of passing through $\delta _{\beta} -1$ generic points we may conclude that transversality holds for all such cuspidal curves.

Certainly we may find a cuspidal cubic $v_A$ whose homology class 
is $3L$ for which transversality holds. 
We claim that gluing $v_A$  
with an immersion representing $\beta-3L$ 
(just as in the proof of Proposition \ref{Existenceofagenuinecusp})
produces a genuine cuspidal curve for which the section $\psi$ is 
transverse to the zero section. Indeed, suppose that $(v_{A1}(\tau),z_1(\tau))$ and $(v_{A2}(\tau),z_2(\tau)$ are 
two paths in $\mathcal{M}_{0,1}^{*}(X, 3L)$ such that they pass through $v_A$ at $\tau=0$, and that the 
tangent vectors to $dv_{Ai}(\tau)(z_i(\tau))$  are linearly independent at $\tau=0$ (this is true because 
we assumed transversality for $v_A$). Then consider the perturbed families of maps 
$u_{1,\varepsilon, t, s,\tau}^{\widehat{\R}}$ and $u_{2,\varepsilon, t, s,\tau}^{\widehat{\R}}$. It is easy 
to see that if $\varepsilon$ is small enough then for this perturbed family of maps $du_{i,\tau} (z_i(\tau)+y-z_i(0))$ have linearly 
independent tangent vectors at $\tau=0$ where $y$ is the location of the cusp of the perturbed map when $\tau=0$. 
(Note that $y$ depends on $\varepsilon, t,s$.)
\end{proof}

\section*{Acknowledgements}

The third author is grateful to Aleksey Zinger, Jingchen Niu and Somnath Basu for
helping him understand the relevant parts of \cite{g2p2and3}. He is also grateful
to Jim Bryan for pointing out \cite{Pandh_Gott} which establishes the
crucial fact that genus zero Gromov--Witten invariants on del-Pezzo surfaces are
enumerative.

\end{document}